\documentclass[
a4paper,
eqnright, reqno]{amsart}
\usepackage{graphicx}
\usepackage{hyperref}
\usepackage{color}
\usepackage{amsmath}
\definecolor{black}{gray}{0}
\definecolor{ttgray}{gray}{0.5}
\definecolor{bfred}{rgb}{0.4,0,0}
\definecolor{emgreen}{rgb}{0,0.3,0}
\definecolor{emmagenta}{rgb}{0.6,0.07,0.07}
\definecolor{sfgb}{rgb}{0,0.3,0.3}
\definecolor{mathblue}{rgb}{0,0,0.4}
\usepackage{myamsart}
\usepackage[opta]{optional}

\newtheorem{Problem}[thm]{Problem}

\title[Relationship between a Wiener--Hopf
  and a Riemann--Hilbert problem]{The relationship between a strip Wiener--Hopf problem
  and a line Riemann--Hilbert problem}
\author{Anastasia V. Kisil}    
\address{Cambridge Centre of Analysis, University of Cambridge,
  Wilberforce Road, Cambridge, CB3 0WA, UK }
\email{a.kisil@maths.cam.ac.uk}
 
\begin{document}

 \begin{abstract}
   In this paper the Wiener--Hopf factorisation problem is presented
   in a unified framework with the Riemann--Hilbert
   factorisation. This allows to establish the exact relationship
   between the two types of factorisation. In particular, in the
   Wiener--Hopf problem one assumes more regularity than for the
   Riemann--Hilbert problem. It is shown that Wiener--Hopf factorisation can be
   obtained using Riemann--Hilbert factorisation on certain lines.
\end{abstract} 

\keywords{Wiener--Hopf,  Riemann-Hilbert, Integral Equations}

\maketitle


\section{Introduction} 

The Wiener--Hopf and the Riemann--Hilbert problems are a subject of many
books and articles \cites{Ehr_spit,
  Ehr_Spec,Constructive_review,Spit,Camara_2}. The similarities of the two
techniques are easily visible and have been noted in many
places. Nevertheless, to the author's knowledge there was no
systematic study of the exact relationship of the two methods. To fill
this gap is the purpose of this article.

It has been suggested in~\cite[Chapter~4.2]{bookWH} that the
Wiener--Hopf equation are a special case of a Riemann--Hilbert
equation.  Specifically, the Riemann--Hilbert problem connects
boundary values of two analytic functions on a contour and the
Wiener--Hopf equation is defined on the strip of common analyticity of
two functions. In the simplest case, both methods use the key concept
of functions analytic in half-planes. The additional regularity for
the Wiener--Hopf equation allows to express the solution in more
simple terms than the Riemann--Hilbert equation.  To be well-defined
the Riemann--Hilbert problem requires some additional regularity,
e.g. the coefficients need to be H\"{o}lder continuous on the contour.

In a different book \cite[Chapter~14.4]{Ga-Che} it has been stated
that the Wiener--Hopf equations results from a bad choice of functions
spaces and instead a Riemann--Hilbert equations should be
considered. Confusingly, those Riemann--Hilbert equations are
sometimes referred to as a Wiener--Hopf equations. Historically, there
has been insufficient interaction between the communities using the
Wiener--Hopf and the Riemann--Hilbert methods, this results in obscuring
disagreements in terminology and notations. Such differences if not
reconciled properly have a tendency to widen.

We consider the following problem as an illustration. Given a
function \(F(t)\) on the real axis:
\begin{equation}
  \label{eq:F-example}
  F(t)=\sqrt{\frac{t^2-(\frac{1}{2}i+1)t-\frac{1}{4}i+\frac{3}{4}}
    {t^2-\frac{3}{2}it+1}}\frac{(t-i+2)}{(t+i+\frac{3}{2})},  
\end{equation}
find two factors \(F^+(t)\) and \(F^-(t)\), which have analytic
extensions in the upper and lower half-planes respectively. In this
rare case the factorisation can be obtained by inspection
\begin{align}
  F(t)&= \left(\sqrt{\frac{(t+\frac{1}{2}i-\frac{1}{2})}{(t+\frac{1}{2}i)}}
    \frac{1}{(t+i+\frac{3}{2})} \right)\times
  \left(\sqrt{\frac{(t-i-\frac{1}{2})}{(t-2i)}}(t-i+2)\right)
  \nonumber \\
  \label{eq:=f+tf-t-}&=F^+(t)F^-(t). 
\end{align}
These functions are depicted in Figure~\ref{fig:vis}. We will comment on
this example in both (the Riemann--Hilbert and Wiener--Hopf)
frameworks at the end of this paper.

\begin{figure}[htbp]
\includegraphics[scale=0.6,angle=0]{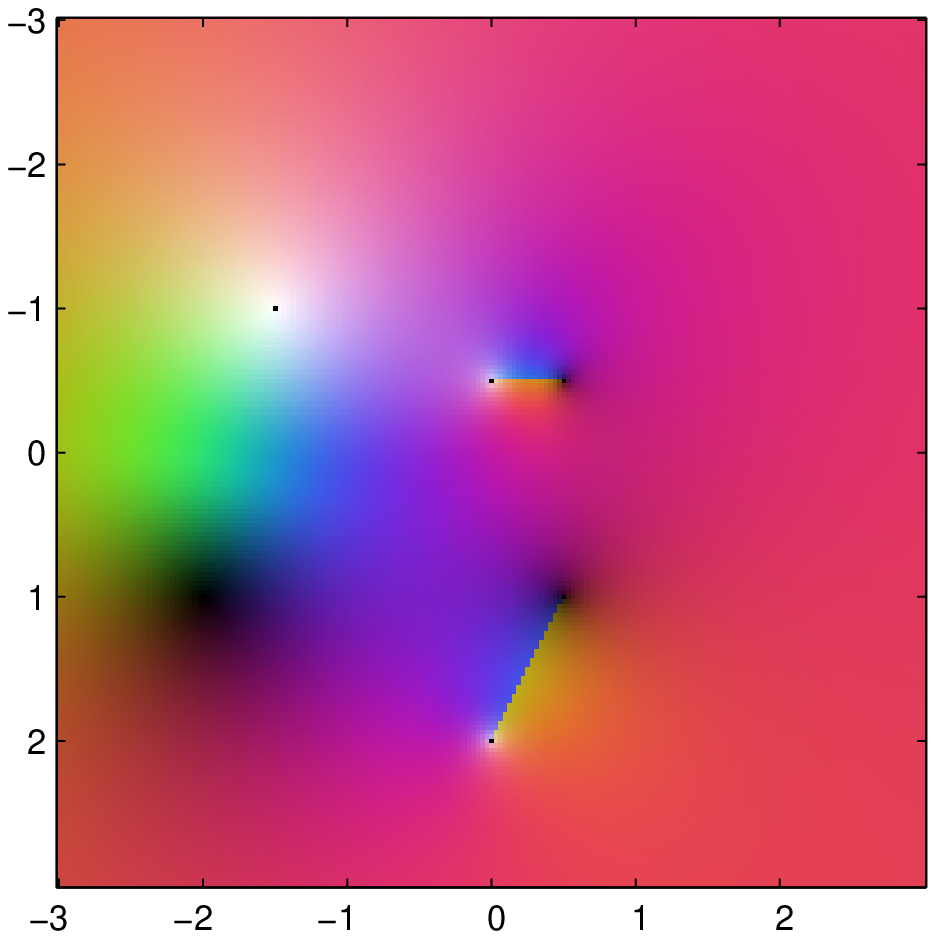}
\includegraphics[scale=0.6,angle=0]{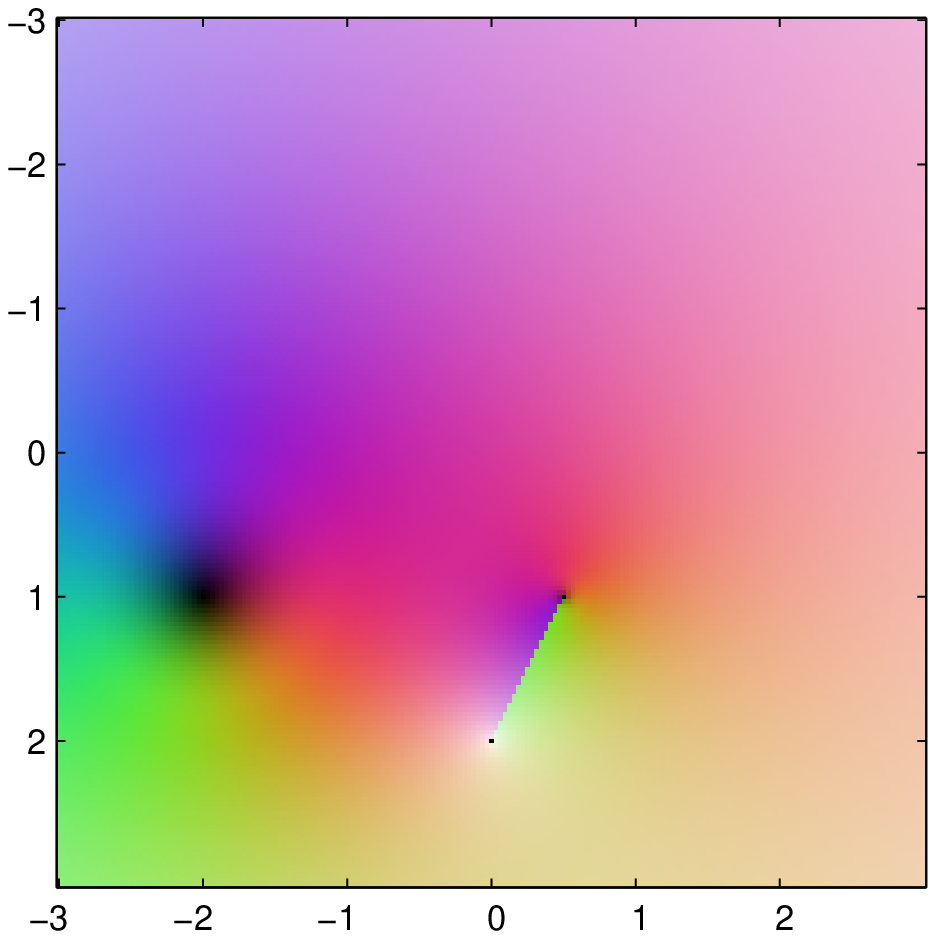}
\includegraphics[scale=0.6,angle=0]{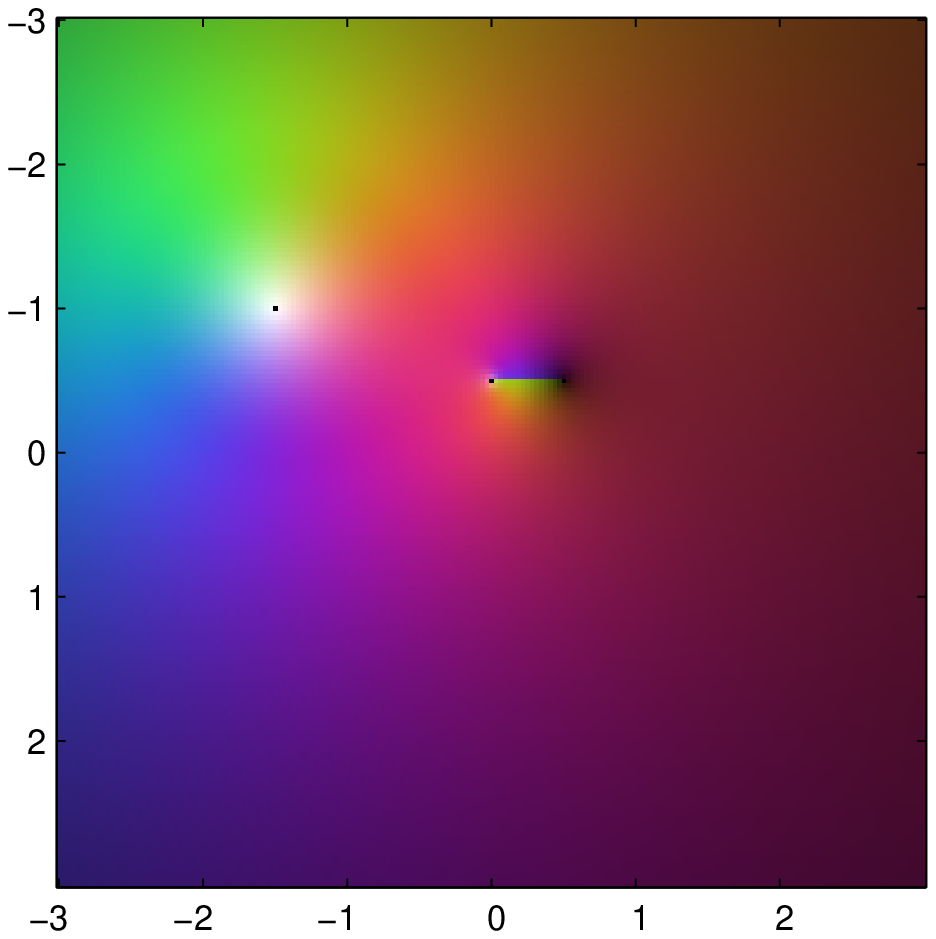} 
  
\caption{A full colour image of functions~\eqref{eq:=f+tf-t-}. The top
  picture shows \(F(z)\) followed by \(F^+(z)\) and \(F^-(z)\). We use
  a colour scheme developed by John Richardson. {\color{red} Red} is
  real, {\color{blue} blue} is positive imaginary, {\color{green}
    green} is negative imaginary, black is small magnitude and white
  is large magnitude.  Branch cuts appear as colour discontinuities.
  Produced using MATLAB package \texttt{zviz.m}. }
 \label{fig:vis} 
\end{figure}

The structure of the paper is as
follows. Section~\ref{sec:Preliminaries} collects some
preliminaries. In Section~\ref{sec:Wiener--Hopf Method}, the usual
theory\opt{optx}{ and applications} of the Wiener--Hopf equation is
recalled. It is presented to highlight the differences with the theory of
the Riemann--Hilbert equations, Section~\ref{sec:R-H}.  Section~\ref{sec:relat-betw-wien} describes at the
relationship between the two methods in the context of integral
equations. In Section~\ref{sec:R-Hv.s} the solution to
Wiener--Hopf and Riemann--Hilbert factorisation is re-expressed in
terms of the Fourier transforms instead of the Cauchy type
integrals. This allows to prove theorems about the exact relationship
of Wiener--Hopf and Riemann--Hilbert factorisation.\opt{optx}{ The
  last section gives some interesting historical remarks.}

\section{Preliminaries} 
\label{sec:Preliminaries}

This section will review important properties of the Fourier transform,
the Cauchy type integral and the relationship between the two. We also states
the generalised version of Liouville's Theorem
together with analytic continuation, in the form which will be used later.

The \emph{Fourier transform} of a function \(f(t)\)  is defined as
\begin{equation}
  \label{eq:fourier}
  F(x)=\frac{1}{\sqrt{2\pi}} \int\limits_{-\infty}^{\infty} f(t) e^{ixt} dt,  \quad  -\infty<x< \infty.
\end{equation}
Throughout this paper the Fourier image and original function will be
denoted by the same letter but they will be upper and lower case
respectively.  The inverse Fourier transform is given:
\[f(t)=\frac{1}{\sqrt{2\pi}} \int\limits_{-\infty}^{\infty} F(x) e^{-ixt} dt,  \quad  -\infty<t< \infty.\]
It is an important fact that if \(f(t)\in L_2(\mathbb{R})\) then
\(F(t)\in L_2(\mathbb{R)}\). This makes the space of square integrable
functions very convenient to work in. Moreover, Fourier transform is an
isometry of \(L_2(\mathbb{R)}\) due to Plancherel's theorem.

The Fourier transform need not be confined to the real axis, as long
as the integral~\eqref{eq:fourier} is absolutely convergent for a
complex \(x\). On an open domain consisting of such parameters \(x\),
the Fourier transform is an analytic function. The Paley--Wiener
theorem (see \eqref{eqn:ana1}--\eqref{eqn:ana3} below) gives the
conditions on the decay at infinity of \(f(t)\) to ensure such
analyticity in different strips and half-planes.


Another key concept is the Cauchy type integral. Let \(F(\tau)\) be
integrable on a simple Jordan curve \(L\), the integral
\[\frac{1}{2\pi i} \int_L \frac{F(\tau)}{\tau -z} d\tau,\]
is called a \emph{Cauchy type integral} and defines an analytic
function on the complement of \(L\).  If \(L\) divides
the complex plane into two disjoint open components then it makes sense
to consider two functions \(F^-\) and \(F^+\) on the respective
domains. In particular, for  the real line we use the notation:
\begin{equation}
\label{eq:hilbert}
 \frac{1}{2\pi i} \int\limits_{-\infty}^{\infty} \frac{F(\tau)}{\tau -z} d\tau= \left\{
 \begin{array}{rl}
 F^{+}(z) & \text{if } \mathrm{Im}\,z>0,\\
  F^{-}(z) & \text{if } \mathrm{Im}\,z<0.\\ 
 \end{array} \right.
\end{equation}
The relationship of the Cauchy type integral to the Fourier transform
is outlined below.
\begin{align}
\label{eq:C-F}
\frac{1}{2\pi i} \int\limits_{-\infty}^{\infty} \frac{F(\tau)}{\tau
  -z} d\tau 
&=\frac{1}{\sqrt{2\pi}} \int\limits_{0}^{\infty} f(t) e^{izt} dt\,,
&& \text{if }\mathrm{Im}\,z>0,\\
\label{eq:C-F2}
\frac{1}{2\pi i} \int\limits_{-\infty}^{\infty} \frac{F(\tau)}{\tau -z} d\tau &=-\frac{1}{\sqrt{2\pi}} \int\limits_{-\infty}^{0} f(t) e^{izt} dt\,,&&
\text{if } \mathrm{Im}\,z<0.
\end{align}
The above integrals are the two ways of arriving at  functions analytic in half-planes.

It was already mentioned that \(L_2(\mathbb{R})\) is a very convenient
with regards to the Fourier transform. The H\"{o}lder continuous
functions produce well-defined boundary values of the Cauchy integral.
Thus, their intersection~\cite{Ga-Che}*{\S~1.2}:
\[ \{\!\{0\}\!\}=L_2(\mathbb{R}) \cap \text{ H\"{o}lder},\] turns out
to be very useful for both the Wiener--Hopf and the Riemann--Hilbert
problems.  The pre-image of \(\{\!\{0\}\!\}\) under the Fourier
transform is denoted \(\{0\}\).

Formulae~\eqref{eq:C-F}--\eqref{eq:C-F2} show the significance of
functions which are zero on the positive or the negative half
lines. Given a function on the real axis we can define the splitting
\begin{equation}
\label{eq:plus}
 f_+(t)= \left\{
 \begin{array}{rl}
 f(t) & \text{if } t>0,\\
  0 & \text{if } t<0,\\ 
 \end{array} \right. 
\quad
 f_-(t)= \left\{
 \begin{array}{rl}
 0 & \text{if } t>0,\\
 -f(t) & \text{if } t<0.\\ 
 \end{array} \right.
\end{equation}
We will say that if \(f \in\{0\}\) then \(f_+(t)\in\{0,\infty\}\) and \(f_-(t)\in\{-\infty,0\}\).

In the rest of the paper we will need to refer to functions that are
analytic on strips or (shifted) half-planes. 
Following~\cite{Ga-Che}*{\S~13}, we define
\(f\in \{a\}\) if \(e^{-ax}f\in \{0\}\), that is a shift in the Fourier
space. Finally, \(f=f_++f_-\in \{a,b\}\) if \(f_+\in \{a\}\) and
\(f_-\in \{b\}\). From the definition of \(f_+\) and \(f_-\) it is
clear that also \(f_+\in \{a, \infty\}\) and \(f_-\in \{-\infty,b\}\).

The Fourier transform of functions in the class \(\{a,b\}\) is denoted
\(\{\!\{a,b\}\!\}\). The celebrated Paley--Wiener theorem states that
the following is equivalent :
\begin{eqnarray}
\label{eqn:ana1}
F_+(z) \text{ analytic in Im}z>a &\iff& F_+(z)\in\{\!\{a,\infty\}\!\},\\
\label{eqn:ana2}
F_-(z) \text{ analytic in Im}z<b &\iff& F_-(z)\in\{\!\{-\infty,b\}\!\},\\
\label{eqn:ana3}
F(z) \text{ analytic in } a<\mathrm{Im}\,z<b &\iff& F(z)\in\{\!\{a,b\}\!\}.
\end{eqnarray}

Recall, a convolution of two functions on the real  line is given by
\begin{equation}
  \label{eq:convolution}
  h(t)=\frac{1}{2\pi i} \int\limits_{-\infty}^{\infty} f(t-s)g(s)\,ds. 
\end{equation}
One of the key properties, which make Fourier techniques useful in
integral equations, is that the Fourier
transform of the convolution is the product of the Fourier transform
of the functions \(f(t)\) and \(g(t)\), i.e.  \(H(x)=F(x)G(x)\).

If \(g\in\{a,b\}\) and \(f\in\{\alpha,\beta\}\) (\(a<b\),
\(\alpha<\beta\)) then \(h(x)\) will be in the class \(\{\text{max}(a,
\alpha), \text{min}(b, \beta) \}\)~\cite{Ga-Che}*{\S~1.3} provided
\[\text{max}(a, \alpha)< \text{min}(b, \beta).\]
After taking the Fourier transform this will become:
\[H(z)=F(z) G(z),\]
and \(H(z) \in\{\!\{\text{max}(a, \alpha), \text{min}(b, \beta)\}\!\} \)
i.e analytic in the strip  
\begin{equation}
\label{eq:strip}
\text{max}(a, \alpha)<\mathrm{Im}\,(z)< \text{min}(b, \beta).
\end{equation}

We will need  a generalised version of Liouville's Theorem
together with analytic continuation: 

\begin{thm}[\cite{Ga-Che}*{\S~3.1}]
  \label{th:generalised-liouville}
  If functions \(F_1(z)\), \(F_2(z)\) are analytic in the upper and lower
  half-planes respectively with exception of \(z_0=\infty,\) \(z_k\)
  (\(k=1, 2 \dots n\)) where they have poles with principal parts:
  \begin{eqnarray}
    G_0(z)&=&c_1^0z+ \dots c_{m_0}^0z^m_0,\\
    G_k\left(\frac{1}{z-z_k}\right)&=& \frac{c_1^k}{z-z_k}+ \dots
    \frac{c_{m_k}^k}{(z-z_k)^{m_k}},
  \end{eqnarray}
  with \(F_1(z)\) and \(F_2(z)\) equal on the real axis, then they define a rational function
  on the whole plane:
  \[F(z)=c+G_0(z)+\sum_1^nG_k\left(\frac{1}{z-z_k}\right),\]
  where \(c\) is an arbitrary constant. Poles  \(z_k\) can lie anywhere
  on the half-planes or on the real axis. 
\end{thm}

\section{The Wiener--Hopf Method}
\label{sec:Wiener--Hopf Method}

In this section the  Wiener--Hopf method is presented in the way it
appears in most papers. This will be revisited to establish the
relationship with the Riemann--Hilbert method. \opt{opta}{For some
  applications of Wiener-Hopf in areas like  elasticity, crack propagation
and acoustics see \cite{Mishuris-Rogosin, owl, Abr_clam, Nigel_cyl}}.  

\opt{optx}{ Wiener--Hopf is an elegant method based on the
  exploitation of the analyticity properties of the functions
  \cite{bookWH}*{Ch. 1}.  The first step of the method is to reduce
  the governing partial differential equation to an ordinary
  differential equation (ODE) via the Fourier transform.  By using the
  ODE and the boundary conditions a Wiener--Hopf equation is
  obtained. The resulting equation is analytic in a strip, which
  mostly contains the real line, and has two unknown functions. Then
  the key step is to obtain a \emph{factorisation} such that the right
  side of the equation is analytic in the upper half-plane and the
  strip, and the left side is analytic in the lower half-plane and the
  strip. Through the use of analytic continuation and Liouville's
  Theorem two separate equations are obtained with one unknown
  function each.  These can now be solved and mapped back to the
  original set-up using the inverse Fourier transform.

The following conventions will be used throughout the section:

\begin{itemize}

\item \emph{A strip} around the real axis (given \(\tau_{-}<0
  <\tau_{+}\)) is \{\(\alpha=\sigma+i\tau\) :
  \(\tau_{-}<\tau<\tau_{+}\)\}\(\subset \mathbb{C}\). \\

\item The subscript \(+\) (or \(-\)) indicates  that the
function is \emph{analytic  in the half-plane} \(\tau>\tau_{-}\) (or \(\tau<\tau_{+}\)). \\

\item  Functions in the Wiener--Hopf equation without the subscript are
  analytic in the strip. \\
\end{itemize}

The Wiener--Hopf problems are recalled below.

The  \emph{multiplicative Wiener--Hopf} problem is: given a function \(K\)
(analytic, zero-free
 and \(K(\alpha) \to 1\) as \( |\sigma| \to \infty\) in the strip) to find functions \(K_+\) and \(K_-\) which satisfy the
following equation in the strip:
\begin{equation}
  \label{eq:mul}
K(\alpha)= K_-(\alpha)K_+(\alpha).
\end{equation}
In addition \(K_-\) and \(K_+\) are required to be:
\begin{itemize}
\item analytic and  non-zero in the respective half-plane. 
\item of \emph{subexponential growth} the respective
  half-planes:
\[|\log K_\pm(\alpha)|=O(|\alpha|^p), \quad p<1,  \quad \text{as} \quad
|\sigma| \to \infty.\]
\end{itemize}
The function \(K(\alpha)\)  in the multiplicative Wiener--Hopf
factorisation will be referred to as  \emph{the kernel}. The functions \(K_+\) and \(K_-\)
will be called \emph{factors} of \(K\).}

\opt{optx}{The multiplicative Wiener--Hopf factorisation can be reduced to the
\emph{additive Wiener--Hopf} problem via application of logarithm. Under the conditions on the
kernel \(K\), let \(f=\log K\) then additive Wiener--Hopf problem is:
given a function  \(f\), find functions \(f_+\) and \(f_-\) which
satisfy the  following equation in the strip:
\begin{equation}
\label{eq:add}
f(\alpha)= f_-(\alpha)+f_+(\alpha).
\end{equation}
Additive and multiplicative problems
are equivalent only in the scalar case. In the matrix Wiener--Hopf
\(K_+\) and \(K_-\) will in general not be  commutative so no such
equivalence is possible.

The existence of solutions to the additive and multiplicative
scalar Wiener--Hopf problem is addressed in Section~\ref{subsec:exact}.}

\opt{optx}{\subsection{The Wiener--Hopf Procedure} 
\label{subsec:Procedure}

The discussion in this section is based on the book by Noble
~\cite{bookWH}. This book is a brilliant introduction to the topic 
and contains numerous examples of the problems which can be solved 
as well as a discussion of the difficulties faced. 

The Wiener--Hopf problem is to find the unknown functions \(\Phi_+\)
and \(\Psi_-\) which satisfy the following equation in \emph{a strip}
around the real axis:
\[A( \alpha)\Phi_+(\alpha)+ \Psi_-(\alpha)+C(\alpha)=0.\] 

The key  step in the Wiener Hopf procedure is to find invertible \(K_+\)
and \(K_-\) (\emph{the Wiener--Hopf factors}) such that:
\[A(\alpha)=K_+(\alpha)K_-(\alpha).\]
 The product of
functions analytic in the
respective half planes will be refereed to as a \emph{factorisation of the
kernel}. 

Then multiplying though:
\[K_+ \Phi_+ ( \alpha)+ K_-^{-1}  \Psi_-(\alpha)+ K_-^{-1}C(\alpha)=0.\]
Now splitting additively:
\[ K_-^{-1}C(\alpha)=C_-(\alpha)+C_+(\alpha).\]
Finally rearranging the equation:
\[K_+ \Phi_+ ( \alpha) +C_+( \alpha)=-K_-^{-1}
\Psi_-(\alpha)-C_-(\alpha).\] The right hand side is a function
analytic in \(\tau>\tau_{-}\) and the left in \(\tau<\tau_{+}\). So by
analytic continuation (from the strip) define a function \(J(\alpha)\)
analytic everywhere in the complex plane. The requirement on the
equation is that there exists an \(n\) such that (\emph{at most
  polynomial growth}):
\[J(\alpha)<|\alpha|^{n} \quad |\alpha| \to \infty.\] Applying the
extended Liouville's theorem to \(J(\alpha)\) implies that
\(J(\alpha)\) is a polynomial of degree less than or equal to
\(n\). So the solution is reduced to finding \(n\) unknown constants
which can be determined in some other way (typically \(n=0\) or
\(1\)). Hence \(\Phi_+ ( \alpha)\) and \(\Psi_-(\alpha)\) are
determined.


There are a few issues which  determine whether the method will
work or not. The first is finding \(K_+\) and \(K_-\). The second
is to find \(C_-(\alpha)\) and \(C_+(\alpha)\) which will turn out to
be similar to the first part.
Lastly  the resulting function can have at most
polynomial growth. 

\begin{rem}
The equation does not have to hold in  the strip
\(\tau_{-}<\tau<\tau_{+}\) the above method works as long as there
is strip which for \(\sigma\) large enough includes the real
line. Typically the strip is deformed around singularities.
\end{rem}

\begin{rem}
In the same way the following equation could be considered: 

\[A( \alpha)\Phi_+(\alpha)+ B(\alpha) \Psi_-(\alpha)+C(\alpha)=0.\] 

As long as \(B(\alpha)\) is invertible. 
\end{rem}
\begin{rem}
The same works for matrix valued functions provided the factorisation
is found. 
\end{rem}
\opt{opta}{ For details on the fascinating history of
Wiener--Hopf see \cite{history}.}}


The next theorem provides a constructive existence theorem for the
additive Wiener--Hopf decomposition (in terms of a Cauchy type integral).

\begin{thm}[\cite{bookWH}*{Ch. 1.3}]
  \label{th:split} 
  Let \(f(\alpha)\) be a function of variable \(\alpha=\sigma+i\tau\),
  analytic in the strip \(\tau_{-}<\tau<\tau_{+}\), such that
  \(f(\sigma+i\tau)<C |\sigma|^{-p}\), \(p>0\) as \(|\sigma| \to
  \infty\), the inequality holding uniformly for all \(\tau\) in the
  strip \(\tau_{-}+ \epsilon<\tau<\tau_{+}-\epsilon,\)
  \(\epsilon>0\). Then for \(\tau_{-}<c<\tau<d<\tau_{+},\)

  \[f(\alpha)= f_-(\alpha)+f_+(\alpha),\]
  \begin{equation}
    \label{eq:W-H-J}
    f_-(\alpha)=- \frac{1}{2\pi i} \int\limits_{-\infty+id}^{\infty+id}
    \frac{f(\zeta)}{\zeta-\alpha} d \zeta \quad ; \quad f_+(\alpha)=\frac{1}{2\pi i} \int\limits_{-\infty+ic}^{\infty+ic}
\frac{f(\zeta)}{\zeta-\alpha} d \zeta,
\end{equation}

where \(f_-\) is analytic in \(\tau<\tau_+\) and \(f_+\) is analytic in \(\tau>\tau_-\).
\end{thm}

We wish to highlight the simplicity of demonstration of this
result. Indeed, to prove (\(\ref{eq:W-H-J}\)) one applies Cauchy's
integral theorem to the rectangle with vertices \(\pm a+ic\), \(\pm
a+id\). From the assumption as regards to the behaviour of
\(f(\alpha)\) as \(|\sigma| \to \infty\) in the strip, the integral on
\(\sigma=\pm a\) tends to zero as \(a\to \infty\) and the result
follows.

The next theorem is a useful variation of the previous theorem,
obtained by taking logarithms. This provides a way to achieve the
multiplicative Wiener--Hopf factorisation.

\begin{thm}~\cite{bookWH}*{Ch. 1.3}
 \label{th:log} 
If \(\log K\) satisfies the conditions for theorem \ref{th:split} in particular that
\(K(\alpha)\)  is analytic and non-zero in the strip and
\(K(\alpha) \to 1\) as \( |\sigma| \to \infty\) then
\(K(\alpha)=K_+(\alpha)K_-(\alpha)\) and \(K_+\), \(K_-\) are
analytic, bounded and non-zero when \(\tau>\tau_{-}\),
\(\tau<\tau_{+}\) respectively.
\end{thm}

The above factors are unique up to
a constant~\cite{Kranzer_68}. In other words, if there are two such factorisations
\(K=K_+K_-\) and \(K=P_+P_-\), then:
\[K_+=cP_+ \quad \text{and} \quad K_-=c^{-1}P_-,\]
where \(c\) is some complex constant. This can be seen by
applying analytic continuation to \(K_+/P_+\) and \(P_-/K_-\) and then using
the extended Liouville theorem~\ref{th:generalised-liouville}.

\opt{optx}{\subsection{Application to PDEs}

In this section it is shown how PDEs with semi-infinite boundary
conditions can be reduced to the Wiener--Hopf equation. There are a lot of
different applications in areas like  elasticity, crack propagation
and acoustics \cite{Mishuris-Rogosin, owl, Abr_clam}.   

Below are some properties of the half space Fourier transform which
will be needed to arrive at an equation of the form in the previous
subsection. The following properties will be needed:
\[ F_+(\alpha)= \int\limits_0^{\infty} f(x) e^{i \alpha x} d x,\]
and
\[ F_-(\alpha)= \int^0_{-\infty} f(x) e^{i \alpha x} d x.\]
In other words \(F_+(\alpha)\) is analytic in the upper half plane and
\(F_-(\alpha)\)  is analytic in the lower half plane; this is true if:

\begin{itemize}
\item \(f(\xi)\) has a finite number of discontinuities on the
real line, has finite number of maxima and minima on any interval
\item \(f(\xi)\) is bounded except at the finite number of points and
  the limit of \(\int f(\xi)\) on the deformed contour as the
  deformation tends to the real line exists.
\end{itemize}


In our case the above will always be satisfied. 

The first example of such PDE comes from acoustics.
Consider the boundary value problem  governed by the
reduced wave equation:
\[\left(\frac{\partial}{\partial x^2}+\frac{\partial}{\partial y^2}
+c^2 \right)
\phi (y, x) =0,  \quad \text{with} \quad c \in \mathbb{R},\]
and in the semi-infinite region \(-\infty
<x< \infty \), \(y\ge 0\) with boundary conditions:
\begin{eqnarray}
  \label{eq:f}
  \phi &=& f(x) \quad \text{on} \quad y=0 \qquad 0 <x< \infty, \\
\frac{\partial  \phi}{\partial y} &=& g(x) \quad \text{on} \quad y=0 \quad -\infty <x< 0.
\end{eqnarray}
Define half-range and full-range Fourier transforms:
\begin{align}
\Phi(y,  \alpha)=& \int\limits_{-\infty}^{0}\phi (y, \xi) e^{i \alpha \xi} d \xi
+\int^{\infty}_{0}\phi(y, \xi)  e^{i \alpha \xi} d \xi,\\
 =& \Phi_{-}(y, \alpha)+\Phi_{+}(y, \alpha),
\end{align}
Taking the Fourier transform of the reduced wave equation gives:
\[\frac{d \Phi (y, \alpha)} {d y^2} - ( \alpha ^2 - c^2) \Phi(y, \alpha)=0.\]
Now solving the ODE:
\[\Phi(y, \alpha)= A(\alpha) e^{- \gamma y} \quad \text{where} \quad \gamma (\alpha)= \sqrt{\alpha^2 - c^2}\]
imposing a restriction of polynomial  growth as \( y \to \infty\).

The taking the Fourier transform of the boundary conditions yield:
\[\Phi_+(0, \alpha)+\Phi_-(0, \alpha)= A(\alpha), \qquad \Phi '_+(0,
\alpha)+\Phi '_-(0, \alpha)= - \gamma (\alpha) A(\alpha).\]
Note that the functions \( \Phi '_-(0, \alpha)\) and \( \Phi_+(0,
\alpha)\) are unknown while the other two are given by the boundary
conditions.

 Now eliminating \( A(\alpha)\)  gives a Wiener--Hopf
equation:
\[- \gamma (\alpha) \{\Phi_+(0, \alpha)+\Phi_-(0, \alpha)\}=\Phi '_+(0,
\alpha)+\Phi '_-(0, \alpha).\]
This can now be solved using the Wiener--Hopf procedure; in other words
\( \Phi '_-(0, \alpha)\) and \( \Phi_+(0,
\alpha)\) can be determined. This allows the \(A(\alpha)\) to be
calculated and hence \(\Phi(y, \alpha)\). Taking the
inverse Fourier transform of \(\Phi(y, \alpha)\)  gives the required solution to the PDE.

\subsection{Random Matrix Theory and  Interacting Particles}

This section demonstrates the variety of applications of the Wiener--Hopf 
method and illustrates the connection to stochastic processes. 
For this purpose the background of the derivation of the Wiener--Hopf 
equation is discussed. But since the actual solution of the equation is not of
 importance here it will be omitted. 
 
 We will explore the fascinating relationship
between random matrices, diffusion processes, interacting particles,
McKean-Vlasov equations and the Wiener--Hopf type equation. The material 
presented here is taken from two paper by Chan (1992) and (1994) \cite{dif1} and \cite{dif2}.

All of this is partly motivated by the beautiful result called the
Wigner semicircle law. Consider a random symmetric matrix with
i.i.d. entries normally distributed with mean 0 and variance \(\sigma^2
n^{-1}\). Then the result states that the density of eigenvalues of
the matrix will be tending to a semi circle with density:
\[\frac{\sqrt{4 \sigma^2-y^2}}{2 \pi \sigma^2}, \quad |y| \le
2 \sigma.\]
This result is not just a curious phenomena this result is important in the 
statistical theory of energy levels. In this setting the random symmetric matrix 
is the finite dimensional approximation to the Hamiltonian.  In the consideration 
of the Schr\"{o}dinger equation the eigenvalues play a big role hence the application 
of this result.
To clarify the nature of this elegant result we will be considering symmetric
matrices with entries diffusion processes in particular the
Ornstein-Uhlenberck processes. 

Let \(X_t\) be a symmetric \(n \times
n\) matrix process evolving according to Ornstein-Uhlenberck:
\[dX=- \frac{1}{2}Xdt+ \frac{1}{2 \sqrt{n}}(dB+dB^T),\]
where \(B\) is standard matrix valued Brownian motion and \(B^T\) its
transpose. It is important to note that this process has normal 
distribution as invariant measure.

Now if we consider the evolution of eigenvalues of this
matrix we will find that they satisfy:
\[d \lambda_i= \frac{1}{\sqrt{n}} dB_i- \frac{1}{2} \lambda_i dt+
\frac{1}{2n} \sum_{k \ne i} \frac{1}{\lambda_i- \lambda_k} dt, \quad
i=1, \dots ,n.\]
where \(B_i\) are independent Brownian motions in \(\mathbb{R}\). 
This can be derived using It\^{o}'s calculus and the invariance of d\(B\) under \(O(n)\).
This can also be interpreted as a system of \(n\) particles
interacting. More precisely we can think of it as  \( n\) charged particles
in a single-well quadratic potential and interacting via electrostatic
repulsion with some random element.

We will be interested in time evolution of a large number of particles. In
particular,  the associated measure valued process:
\[ \mu_n(t)=\frac{\sum_{i=1}^n \delta_{\lambda_i(t)}}{n},\]
where \(\delta _y\) denotes a point mass at \(y\). We note that the flow of measure 
is non-linear  and highly singular.
Nevertheless it has been shown that \(\mu_n(t)\) converges to a deterministic
measure-valued function \(\mu_t\).
What is more it can be shown that if \(\mu_t\) has no atomic component,
then \(\mu_t\) is the weak solution of the McKean-Vlasov equation:
\[\frac{d}{dt} \int \phi (\lambda) \mu_t (d \lambda)= \frac{1}{2} \int \left(
\int \frac{ \mu_t (d y)}{ \lambda- y}-\lambda \right)
\phi'(\lambda)\mu_t (d \lambda), \]
for the suitable test functions \( \phi\), one such choice would be to take \(\phi_{\theta} (\lambda)=e^{-i \theta \lambda}\), \(\theta \in \mathbb{R}\). 
Note that the integral in the bracket is the Cauchy principle value. 

It has been shown by Chan that the Wigner semi circle density (with \(\sigma=\sqrt(2)/2\)) is the unique equilibrium point of the 
McKean-Vlasov equation under the regularity conditions of having H\"{o}lder-continuous density and belonging to \(L^2\). 
Additionally if \(\mu_0\) possesses finite moments of all order then \(\mu_t  \) actually tends to the semi circle density. 

But there are difficulties in analysing the flow of measure \(\mu_t\)
using the McKean-Vlasov equation. They are only a
 weak solution so there is always the business  of considering appropriate test functions. Below we describe a way to get around this problem.

We set the test functions to be  \(\phi_{\theta} (\lambda)=e^{-i \theta \lambda}\), \(\theta \in \mathbb{R}\). Then  the Fourier transform of \( \mu\) called \(\nu\) will
satisfy the following integral Wiener--Hopf type equation:
\[\frac{\partial \nu}{\partial t}= \frac{\theta}{4} \left( \int\limits_{-\infty}^0
\nu_t(u) \nu_t(\theta -u) du - \int^{\infty}_0
\nu_t(u) \nu_t(\theta -u) du \right) - \frac{\theta}{2}\frac{\partial
  \nu}{\partial \theta},\]
where:
\[ \nu_t(\theta)= \int _{-\infty}^{\infty} e^{-i \theta \lambda } \mu_t(d \lambda).\]
The study of this equation will help to clarify the semi circle law. This concludes this section 
which hopefully demonstrated some  deep links that the Wiener--Hopf equation may offer.

}
\section{The Riemann--Hilbert Method}
\label{sec:R-H}

In this section the main differences from the Wiener--Hopf method are
presented.\opt{optx}{ The same procedure as in
  Subsection~\ref{subsec:Procedure} follows for the Riemann--Hilbert
  equation.} They are rooted in the conditions imposed on the
functions and the resulting formula. For applications of the
Riemann--Hilbert problem the interested reader is refereed
to~\cites{Fokas_Painleve,Gah}.

The Riemann--Hilbert Problem can be stated as follows:

\begin{Problem}[Riemann--Hilbert]
  On the real line two functions are given, \(H(t)\) and non-zero
  \(D(t)\)  with
  \begin{displaymath}
    D(t)-1\in \{\!\{0\}\!\}\qquad  \text{ and } \qquad H(t) \in \{\!\{0\}\!\}.
  \end{displaymath}
  It is required to find two functions \(F^{\pm}(z)\) analytic in the
  upper and lower half-planes such that their boundary values
  \(F^{\pm}(t)\) on the real line satisfy two conditions:
  \begin{enumerate}
  \item  \(F^{\pm}(t)\) belong to the classes \(\{\!\{0,\infty\}\!\}\) and
    \(\{\!\{-\infty, 0\}\!\}\);
  \item   The identity holds:
    \begin{equation}
    \label{eq:R-H}
    D(t)F^-(t)+H(t)=F^{+}(t), \qquad \text {for all real } t.
  \end{equation}
  \end{enumerate}
\end{Problem}

As it has been seen for the Wiener--Hopf method the key steps in the
solution is the additive splitting or jump problem (Thm.~\ref{th:split}) and the
factorisation (Thm.~\ref{th:log}). First splitting is addressed as before:
\begin{Problem}[Jump problem]
  \label{prob:jump-probl}
  On the real line a function \(F(t)\in \{\!\{0\}\!\}\) is
  given. It is required to find two functions \(F^{\pm}(z)\) analytic
  in the upper and lower half-planes with boundary functions on the
  real line belonging to the classes \(\{\!\{0,\infty\}\!\}\) and
  \(\{\!\{-\infty, 0\}\!\}\) and satisfying:
\begin{equation*}
  F(t)=F^{+}(t)+F^-(t),
\end{equation*}
on the real line.  
\end{Problem}

The solution is offered by the Sokhotskyi--Plemelj formula:
\begin{eqnarray}
 \label{eq:S-P}
 F^{+}(x)-F^-(x)=F(x), \qquad F^{+}(x)+F^-(x)=\frac{1}{\pi i}
 \int\limits_{-\infty}^{\infty} \frac{F(\tau)}{\tau -x}\, d\tau .
\end{eqnarray}
We shall note,  that the derivation of~\eqref{eq:S-P} is more involved
than the proof of Thm.~\ref{th:split}.

Next the factorisation problem is examined. The \emph{index}  of a continuous non-zero
function \(K(t)\) on the real line  is:
\begin{equation}
  \label{eq:eq:index-func-defn}
  \mathrm{ind}(K(t))=\frac{1}{2\pi}\big(
  \lim_{t\to+\infty}\arg  K(t)-\lim_{t\to-\infty}\arg
  K(t)\big).
\end{equation}
In other words the index is the winding number of the curve \((\text{Re } K(t),
\mathrm{Im}\, K(t)) \) \(t \in \mathbb{R}\). 
Note that
\(\mathrm{ind}\, \frac{t-i}{t+i}=1.\)
Given a function \(K(t)\)  with index \(\kappa\) one can reduce it to zero index
by considering a new function
\begin{displaymath}
  K_0(t)=K(t)\left( \frac{t-i}{t+1}\right) ^{-\kappa}.
\end{displaymath}
The reason it is important to consider functions \(K(t)\) with zero
index is to ensure \(\ln K(t)\) is single-valued.  For the rest of this
paper we will assume that all functions have zero index.  Taking logarithms and applying
the Sokhotskyi--Plemelj formula we get a solution to the factorisation problem~\cite{Ga-Che}:

\begin{Problem}[Riemann--Hilbert Factorisation]
  \label{prob-R-H-factor}
  Let a non-zero function \(K(t)\), such that  \(K(t)-1\in
  \{\!\{0\}\!\}\) and \(\mathrm{ind}\, K(t)=0\), be   given. It is required to find two
  functions \(K^{\pm}(z)\) analytic in the upper and lower half-planes
  with boundary functions \(K^{\pm}(t)\) on the real line belonging to the
  classes \(\{\!\{0,\infty\}\!\}\) and \(\{\!\{-\infty, 0\}\!\}\) and
  satisfying:
  \begin{equation*}
    K(t)=K^{+}(t)K^-(t),
  \end{equation*}
  on the real line.
\end{Problem}

\opt{optx}{ We
can normalise so that \(K_{\pm}(t) \to 1\) for \(t \to \pm
\infty\). The Riemann--Hilbert problem will be returned to in
Section~\ref{sec:R-Hv.s}.}

\section{Relationship Between Wiener-Hopf and Riemann-Hilbert via  Integral Equations}
\label{sec:relat-betw-wien}

This section demonstrates how one type of integral equation can either lead
 to the Wiener--Hopf problem or the Riemann--Hilbert problem,
depending on the class of function where the solution is
sought. Integral equations have historically motivated the
introduction of the Wiener--Hopf equation~\cite{history}\opt{optx}{, see
  Section~\ref{sec:his}}.

In applications, a time-invariant process can be modelled by an
integral equations with convolution on the half-line:
\begin{equation}
  \label{eq:conv-half-line}
   \int\limits_0^{\infty} k(x-y)f(y)\, dy=g(x), \quad 0<x< \infty.
\end{equation}
Here the kernel \(k(x-y) \in \{a,b\}\) represents the process, 
\(g(x)\) is a given output and
\(f(y)\) is an input to be determined. To solve the equation~\eqref{eq:conv-half-line}
we complement the domain of \(x\):
\begin{equation}
  \label{eq:other-half-line}
  \int\limits_0^{\infty} k(x-y)f(y) dy=h(x), \quad -\infty<x< 0,
\end{equation}
where \(h(x)\) is unknown.  Then, by applying the Fourier
transform we get the equation:
\begin{equation}
  \label{eq:int}
 F_+ (\alpha)K(\alpha)-G_+(\alpha)= H_-(\alpha).
\end{equation}

Now it is time to examine the above equation more carefully and in 
particular clarify the analyticity regions. There will be two
different cases considered: the first one will be a typical example
from applications and the second the most general solution. We will see
that  the former will lead to a Wiener--Hopf equation and the latter to
a Riemann--Hilbert equation.


\begin{enumerate}
\item In equation~(\ref{eq:int}) \(K(\alpha)\) and \(G_+(\alpha)\) are
  known, thus their region of regularity can be determined. From the
  maximal growth rate of \(f(x)\) as \(x \to +\infty\) and \(h(x)\) as
  \(x \to -\infty\) the analyticity half-planes are determined as
  in~\eqref{eqn:ana1}--\eqref{eqn:ana3}. For example, in the integral
  equation for Sommerfeld's half-plane
  problem~\cite{bookWH}*{Ch. 2.5}, the known functions are in the
  following classes:
  \[G_+(\alpha)\in\{\!\{a \cos\theta, \infty\}\!\}, \quad
  K(\alpha)\in\{\!\{-a,a\}\!\}, \] and the unknown are in:
\[F_+(\alpha)\in\{\!\{a \cos\theta,\infty\}\!\},\quad
H_-(\alpha)\in\{\!\{-\infty,a\}\!\}.\]
Here \(a\) and \(\theta\) are some constant.
From~\eqref{eq:strip}, \(K(\alpha)F_+(\alpha)\in\{\!\{a \cos\theta,a\}\!\}\)
 and Equation~(\ref{eq:int}) holds in the strip
\(a \cos\theta< \mathrm{Im}\,
\alpha <a\)
. 
We obtained a Wiener--Hopf equation.

\item We considered \(K(\alpha) \in \{\!\{-a,a\}\!\}\) and the other
  functions will be assumed to belong to the largest possible class
  for (\ref{eq:int}) to be solvable. If
  \(F_+(\alpha)\in\{\!\{b,\infty\}\!\}\) the convolution will exist
  for \(b\le a\), so we will take equality as the minimal condition on
  regularity. Then, the maximal class \cite{Ga-Che}, in which the
  integral equation has a solution, is:
\[G_+(\alpha)\in\{\!\{a, \infty\}\!\}, \quad F_+(\alpha)\in\{\!\{a,\infty\}\!\},\quad
H_-(\alpha)\in\{\!\{-\infty,a\}\!\}.\]
Furthermore, from~(\ref{eq:strip}) we have
\(K(\alpha)F_+(\alpha)\in\{\!\{a ,a\}\!\}\). Hence, (\ref{eq:int}) is
only valid on a line \(\mathrm{Im}\, \alpha =a\) and it is a
Riemann--Hilbert equation.
\end{enumerate}

The above shows that the Wiener--Hopf equation is the result
of a better regularity of the function at infinity than is minimally needed for
a solution to exist. 

\section{Relationship between the Wiener--Hopf and the Riemann-Hilbert
  Equations}
 \label{sec:R-Hv.s}

To examine the relationship between the Wiener--Hopf and the Riemann--Hilbert
equations we will restate problems for the same class of functions and
re-express the solution in terms of the Fourier integrals instead of
Cauchy type integrals.

We begin with the jump problem in the Riemann--Hilbert case: 
\begin{thm}
  \label{thm:R-H}
  The solution to Jump Problem~\ref{prob:jump-probl}  can be expressed as:
  \begin{equation}
    \label{eq:R-H-jump-sol}
    F^{+}(z)= \frac{1}{\sqrt{2\pi}} \int\limits_{0}^{\infty} f(t) e^{izt} \,dt
    \quad     F^-(z)=-\frac{1}{\sqrt{2\pi}} \int\limits_{-\infty}^{0} f(t)
    e^{izt} \,dt, 
  \end{equation}
  where \(f(t)\) is the inverse Fourier transform of \(F(t)\).
\end{thm}

\begin{proof}
  The solution to the problem is given by the Sokhotskyi-Plemelj
  formula~(\ref{eq:S-P}). We re-express them in terms of the Fourier
  integrals using~(\ref{eq:C-F}):
\begin{eqnarray*}
 F(x)&=&\frac{1}{\sqrt{2\pi}} \int\limits_{-\infty}^{\infty} f(t) e^{ixt}
  \,dt,\\
&=&\frac{1}{\sqrt{2\pi}} \int\limits_{0}^{\infty} f(t) e^{ixt}
  \,dt+
\frac{1}{\sqrt{2\pi}} \int\limits_{-\infty}^{0} f(t)e^{ixt}
  \,dt,\\ 
&=&F^{+}(x)-F^-(x)
\end{eqnarray*}
and
\begin{eqnarray*}
\frac{1}{\pi i} \int\limits_{-\infty}^{\infty} \frac{F(\tau)}{\tau -x}
d\tau&=&F^{+}(x)+F^-(x),\\
&=&\frac{1}{\sqrt{2\pi}} \int\limits_{0}^{\infty} f(t) e^{ixt}
  \,dt-
\frac{1}{\sqrt{2\pi}} \int\limits_{-\infty}^{0} f(t)e^{ixt},\\
&=&\frac{1}{\sqrt{2\pi}} \int\limits_{-\infty}^{\infty}\text{sgn }(t) f(t) e^{ixt}
  \,dt.
\end{eqnarray*}
In other words, the Sokhotskyi-Plemelj formula for the real line in
terms of the Fourier transform is:
\begin{equation}
  \label{eq:eq:S-P_2}
F^{+}(x)-F^-(x)=F(x), \quad F^{+}(x)+F^-(x)=\frac{1}{\sqrt{2\pi}} \int\limits_{-\infty}^{\infty}\text{sgn }(t) f(t) e^{ixt}
 \, dt,
\end{equation}
From this the result follows.
\end{proof}

\begin{rem}
  The problem does not change if it is shifted by real \(a\) to the classes
\(F(t)\in \{\!\{a\}\!\}\) with \(\{\!\{a,\infty\}\!\}\) and \(\{\!\{-\infty, a\}\!\}\).
\end{rem}

To clarify the relation between two problems we re-state the
Wiener--Hopf jump problem from Thm.~\ref{th:split} as follows:
\begin{Problem}[Wiener--Hopf jump problem]
  \label{prob:W-H-jump}
  On a strip \(S=\{z: a<\mathrm{Im}\,(z)<b\}\) a function \(F(z)\in
  \{\!\{a,b\}\!\}\) is given. It is required to find two functions
  \(F^{\pm}(z)\) analytic in the half-planes \(\{z: a<\mathrm{Im}\,(z)\}\)
  and  \(\{z:
  \mathrm{Im}\,(z)<b\}\) respectively,
  which belong to the classes \(\{\!\{a,\infty\}\!\}\) and
  \(\{\!\{-\infty, b\}\!\}\) and satisfying:
  \begin{equation*}
    F(z)=F^{+}(z)+F^-(z),
  \end{equation*}
  on the strip \(S\).  
\end{Problem}
Similarly to Thm.~\ref{thm:R-H} we find:
\begin{thm}
The solution of Prob.~\ref{prob:W-H-jump} can be expressed as:
\begin{equation*}
 F^{+}(z)= \frac{1}{\sqrt{2\pi}} \int\limits_{a}^{\infty} f(t) e^{izt}\, dt
 \quad     F^-(z)=\frac{1}{\sqrt{2\pi}} \int\limits_{-\infty}^{b} f(t)
 e^{izt}\, dt, 
\end{equation*}
where \(f(t)\) is the inverse Fourier transform of \(F(t)\).
\end{thm}

\begin{proof}
  To prove the theorem we use the solution~\eqref{eq:R-H-jump-sol} of the
  Riemann--Hilbert problem first for the line \(\mathrm{Im}\,(z)=a\)
  and then for \(\mathrm{Im}\,(z)=b\). On the line
  \(\mathrm{Im}\,(z)=a\) we obtain:
\begin{equation}
  F(x+ia)=F^{+}(x+ia)+F^-(x+ia),
\end{equation}
with the factors given by:
\begin{equation}
 F^{+}(z)= \frac{1}{\sqrt{2\pi}} \int\limits_{a}^{\infty} f(t) e^{izt} \,dt
 \quad \text{ and } \quad     F^-(z)=-\frac{1}{\sqrt{2\pi}} \int\limits_{-\infty}^{a} f(t)
 e^{izt} \,dt. 
\end{equation}
Note that:
\begin{equation}
  F(x+ia)-F^-(x+ia)=F^{+}(x+ia),
\end{equation}
where the left-hand side has continuous analytic extension into the
strip \(S\), thus the same will be true for \(F^{+}(x+ia)\). This allows to move
to the other side of the strip:
\begin{equation}
  F(x+ib)=F^{+}(x+ib)+F^-(x+ib).
\end{equation}
Now, an application of the Sokhotskyi-Plemelj formula gives:
\begin{equation}
 F^{+}(z)= \frac{1}{\sqrt{2\pi}} \int\limits_{b}^{\infty} f(t) e^{izt} \,dt. 
\end{equation}
In other words, the analyticity of \( F^{+}(z)\) was extended to the
strip \(S\) and the result follows.
\end{proof}

\begin{rem}
  \label{rem:1}
  Note, that in the above proof one could have chosen initially the
  \(\mathrm{Im}\,(z)=b\) and then extended function \(F^-(x+ib)\) to
  the line \(\mathrm{Im}\,(z)=a\). In fact, any line in between
  \(\mathrm{Im}\,(z)=c\) with \(a<c<b\) could have been taken and a
  solution of the Riemann--Hilbert problem obtained. Then, the
  existance of analytical extension of both functions \(F^{+}\) and
  \(F^{-}\) can be shown in a similar manner.
\end{rem}

Similarly, we describe the relationship between the factorisation in
the  Wiener--Hopf and Riemann--Hilbert setting. 

\begin{thm}
\label{thm:M-R-H}
 The solution or Prob.~\ref{prob-R-H-factor} can be expressed as:
\begin{equation*}
 K^{+}(z)= \exp\left(\frac{1}{\sqrt{2\pi}} \int\limits_{0}^{\infty} \kappa(t) e^{izt} \,dt\right)
 \quad\text{and}\quad
     K^-(z)=\exp\left(-\frac{1}{\sqrt{2\pi}} \int\limits_{-\infty}^{0} \kappa(t)
 e^{izt} \,dt\right), 
\end{equation*}
where \(\kappa(t)\) is the inverse Fourier transform of \(\ln K(t)\).
\end{thm}

\begin{proof} From the \(\mathrm{ind}\,K(t)=0\) it follows that \(\ln K(t)\)
  is single-valued. An application of Thm.~\ref{thm:R-H} to \(\ln
  K(t)\) yelds:
\[\ln K(t)= \frac{1}{\sqrt{2\pi}} \int\limits_{0}^{\infty} \kappa(t) e^{izt} \,dt+\left(-\frac{1}{\sqrt{2\pi}} \int\limits_{-\infty}^{0} \kappa(t)
 e^{izt} \,dt\right).\]
The result follows from taking exponents of both sides of the last identity. 
\end{proof}
To express relations between two factorisation problems we formulate
the Wiener--Hopf factorisatiojn in a suitable form.
\begin{Problem}[Wiener--Hopf Factorisation]
  \label{prob:W-H-fact}
  A function \(K(z)\) is non-zero on the whole strip \(S=\{z:
  a<\mathrm{Im}\,(z)<b\}\), furthermore \(K(z)-1\in \{\!\{a,b\}\!\}\)
  and \(\mathrm{ind}\,K(x+ia)=0\). It is required to find two
  functions \(K^{\pm}(z)\) analytic in the half-planes \(\{z: a<\mathrm{Im}\,(z)\}\)
  and  \(\{z:
  \mathrm{Im}\,(z)<b\}\) belonging to the
  classes \(\{\!\{a,\infty\}\!\}\) and \(\{\!\{-\infty, b\}\!\}\)
  respectively, such that:
\begin{equation*}
  K(z)=K^{+}(z)K^-(z),
\end{equation*}
on the strip \(a<\mathrm{Im}\,(z)<b\).
\end{Problem}
Similarly to Thm.~\ref{thm:M-R-H} we find:
\begin{thm}
\label{thm:W-h_2}
 The solution of Prob.~\ref{prob:W-H-fact} can be expressed as:
\begin{equation*}
 K^{+}(z)= \exp\left(\frac{1}{\sqrt{2\pi}} \int\limits_{b}^{\infty} \kappa(t) e^{izt} \,dt\right)
 \quad\text{and}\quad
     K^-(z)=\exp\left(-\frac{1}{\sqrt{2\pi}} \int\limits_{-\infty}^{a} \kappa(t)
 e^{izt} \,dt\right), 
\end{equation*}
where \(\kappa(t)\) is the inverse Fourier transform of \(\ln K(t)\).
\end{thm}

\begin{proof}
The function \(K(z)\) on the line \(\mathrm{Im}\,z=a\)  satisfies all the assumptions of
Thm.~\ref{thm:M-R-H}, thus we obtain:
\[K(x+ia)=K^{+}(x+ia)K^-(x+ia),\]
with:
\begin{equation*}
 K^{+}(z)= \exp\left(\frac{1}{\sqrt{2\pi}} \int\limits_{a}^{\infty} \kappa(t) e^{izt} \,dt\right)
 \quad\text{and}\quad  K^-(z)=\exp\left(-\frac{1}{\sqrt{2\pi}} \int\limits_{-\infty}^{a} \kappa(t)
 e^{izt} \,dt\right). 
\end{equation*}
Since \(\mathrm{ind}\,K(x+ia)=0\) and \(K(z)\) zero free on the strip, it
follows that \(\mathrm{ind}\,K(x+is)=0\) for all \(a<s<b\).
This is again expressed as
\[K(x+ia)K^-(x+ia)=K^{+}(x+ia).\] Because the left hand side has
continuous analytic extension in the strip \(S\), the function
\(K^{+}(x+ia)\) has the extension as well. This gives meaning to the
expression:
\[K(x+ib)=K^{+}(x+ib)K^-(x+ib).\]
Now, \(K(x+ib)\) satisfies all the assumptions of
Thm.~\ref{thm:M-R-H}, thus
\begin{equation*}
 K^{+}(z)= \exp\left(\frac{1}{\sqrt{2\pi}} \int\limits_{b}^{\infty} \kappa(t)
 e^{izt} \,dt\right).
\end{equation*}
 This provides the required factorisation.
\end{proof}

Remark~\ref{rem:1} also holds here. There are some differences in the
formulation of Thm.~\ref{th:log} and Thm.~\ref{thm:W-h_2}. The most
significant difference\footnote{Remarkably, the excellent
  book~\cite{bookWH} does not mention the index of functions at all.}
is the assumption that \(\mathrm{ind}\,K(x+ia)=0\). This is because
Thm.~\ref{th:log} assumes the existence of single valued \(\ln K(t)\).


The above derivations show that the Wiener--Hopf equations is
characterised by extra regularity, namely a domain of
analyticity. Noteworthy, there are further applications of this
feature, for example, a strip deformation.  Consider a Wiener--Hopf
equation:
\[A( \alpha)\Phi_+(\alpha)+ \Psi_-(\alpha)+C(\alpha)=0.\] Assume as
before that \(\Phi_+(\alpha)\) and \(\Psi_-(\alpha)\) are analytic in
the upper or lower half-planes and the strip respectively. However,
assume that this time \(A( \alpha)\) and \(C(\alpha)\) have
singularities in the strip
. 
Then, by taking a subset of the strip 
 the Wiener--Hopf equations can still be solved
\cite{Ab_Kh_n}.

Finally, we are going to revisit the factorisation of
\(F(t)\)~\eqref{eq:F-example}. Instead of spotting the factors by
inspection the Riemann--Hilbert formula can be used, yielding factors
analytic in the half-planes. By inspection of 
singularities in the complex plane, \(F(z)\) has a strip of
analyticity \(-1+\epsilon<\mathrm{Im}\,(z)<1/2-\epsilon\). Hence, it
is also possible to apply the Wiener--Hopf formula to obtain the
factorisation. Due to the uniqueness of the factors and analytic
continuation it follows, that both methods shall produce 
results coinciding with~\eqref{eq:=f+tf-t-} and illustrated by
Figure~\ref{fig:vis}. 

\section{Acknowledgements}

I am grateful for  support from Prof. Nigel Peake. I
benefited from useful discussions with Dr Rogosin. Also suggestions
 of the anonymous referees helped to improve this paper.  This work was supported by the UK Engineering and Physical Sciences Research Council (EPSRC) grant EP/H023348/1 for the University of Cambridge Centre for Doctoral Training, the Cambridge Centre for Analysis.


\bibliography{newgeometry}

\begin{bibdiv}
\begin{biblist}

\bib{Abr_clam}{article}{
      author={Abrahams, I.~David},
      author={Davis, Anthony M.~J.},
      author={Llewellyn~Smith, Stefan~G.},
       title={Matrix {W}iener-{H}opf approximation for a partially clamped
  plate},
        date={2008},
        ISSN={0033-5614},
     journal={Quart. J. Mech. Appl. Math.},
      volume={61},
      number={2},
       pages={241\ndash 265},
         url={http://dx.doi.org/10.1093/qjmam/hbn004},
      review={\MR{2414435 (2009d:74047)}},
}

\bib{Camara_2}{article}{
      author={C{\^a}mara, M.~C.},
      author={dos Santos, A.~F.},
       title={Wiener-{H}opf factorization of a generalized {D}aniele-{K}hrapkov
  class of {$2\times 2$} matrix symbols},
        date={1999},
        ISSN={0170-4214},
     journal={Math. Methods Appl. Sci.},
      volume={22},
      number={6},
       pages={461\ndash 484},
  url={http://dx.doi.org/10.1002/(SICI)1099-1476(199904)22:6<461::AID-MMA45>3.0.CO;2-R},
      review={\MR{1679128 (2000a:47052)}},
}

\bib{Ehr_Spec}{article}{
      author={Ehrhardt, Torsten},
      author={Speck, Frank-Olme},
       title={Transformation techniques towards the factorization of
  non-rational {$2\times 2$} matrix functions},
        date={2002},
        ISSN={0024-3795},
     journal={Linear Algebra Appl.},
      volume={353},
       pages={53\ndash 90},
         url={http://dx.doi.org/10.1016/S0024-3795(02)00288-4},
      review={\MR{1918749 (2003f:15021)}},
}

\bib{Ehr_spit}{article}{
      author={{\`E}rkhardt, T.},
      author={Spitkovski{\u\i}, I.~M.},
       title={Factorization of piecewise-constant matrix functions, and systems
  of linear differential equations},
        date={2001},
        ISSN={0234-0852},
     journal={Algebra i Analiz},
      volume={13},
      number={6},
       pages={56\ndash 123},
      review={\MR{1883840 (2004h:30045)}},
}

\bib{Fokas_Painleve}{book}{
      author={Fokas, Athanassios~S.},
      author={Its, Alexander~R.},
      author={Kapaev, Andrei~A.},
      author={Novokshenov, Victor~Yu.},
       title={Painlev\'e transcendents},
      series={Mathematical Surveys and Monographs},
   publisher={American Mathematical Society, Providence, RI},
        date={2006},
      volume={128},
        ISBN={0-8218-3651-X},
         url={http://dx.doi.org/10.1090/surv/128},
        note={The Riemann-Hilbert approach},
      review={\MR{2264522 (2010e:33030)}},
}

\bib{Ga-Che}{book}{
      author={Gahov, F.~D.},
      author={{\v{C}}erski{\u\i}, Ju.~I.},
       title={Equations of convolution type (in russian)},
}

\bib{Gah}{book}{
      author={Gakhov, F.~D.},
       title={Boundary value problems},
      series={Translation edited by I. N. Sneddon},
   publisher={Pergamon Press, Oxford-New York-Paris; Addison-Wesley Publishing
  Co., Inc., Reading, Mass.-London},
        date={1966},
      review={\MR{0198152 (33 \#6311)}},
}

\bib{Constructive_review}{unpublished}{
      author={Gennady~Mishuris, Sergei~Rogosin},
       title={Constructive methods for factorization of matrix-functions},
        note={Submitted},
}

\bib{Spit}{incollection}{
      author={Gohberg, I.},
      author={Kaashoek, M.A.},
      author={Spitkovsky, I.M.},
       title={An overview of matrix factorization theory and operator
  applications},
    language={English},
        date={2003},
   booktitle={Factorization and integrable systems},
      editor={Gohberg, Israel},
      editor={Manojlovic, Nenad},
      editor={dos Santos, AntцЁnioFerreira},
      series={Operator Theory: Advances and Applications},
      volume={141},
   publisher={Birkhц╓user Basel},
       pages={1\ndash 102},
         url={http://dx.doi.org/10.1007/978-3-0348-8003-9_1},
}

\bib{owl}{article}{
      author={Jaworski, Justin~W.},
      author={Peake, N.},
       title={Aerodynamic noise from a poroelastic edge with implications for
  the silent flight of owls},
        date={2013},
        ISSN={0022-1120},
     journal={J. Fluid Mech.},
      volume={723},
       pages={456\ndash 479},
  url={http://0-dx.doi.org.pugwash.lib.warwick.ac.uk/10.1017/jfm.2013.139},
      review={\MR{3045100}},
}

\bib{Kranzer_68}{article}{
      author={Kranzer, Herbert~C.},
       title={Asymptotic factorization in nondissipative {W}iener-{H}opf
  problems},
        date={1967},
     journal={J. Math. Mech.},
      volume={17},
       pages={577\ndash 600},
      review={\MR{0220020 (36 \#3087)}},
}

\bib{history}{article}{
      author={Lawrie, Jane~B.},
      author={Abrahams, I.~David},
       title={A brief historical perspective of the {W}iener-{H}opf technique},
        date={2007},
        ISSN={0022-0833},
     journal={J. Engrg. Math.},
      volume={59},
      number={4},
       pages={351\ndash 358},
         url={http://dx.doi.org/10.1007/s10665-007-9195-x},
      review={\MR{2373797 (2009a:45002)}},
}

\bib{Mishuris-Rogosin}{article}{
      author={Mishuris, Gennady},
      author={Rogosin, Sergei},
       title={An asymptotic method of factorization of a class of matrix
  functions},
        date={2014},
     journal={Proceedings of the Royal Society A: Mathematical, Physical and
  Engineering Science},
      volume={470},
      number={2166},
  eprint={http://rspa.royalsocietypublishing.org/content/470/2166/20140109.full.pdf+html},
  url={http://rspa.royalsocietypublishing.org/content/470/2166/20140109.abstract},
}

\bib{bookWH}{book}{
      author={Noble, B.},
       title={Methods based on the {W}iener-{H}opf technique for the solution
  of partial differential equations},
      series={International Series of Monographs on Pure and Applied
  Mathematics. Vol. 7},
   publisher={Pergamon Press},
     address={New York},
        date={1958},
      review={\MR{0102719 (21 \#1505)}},
}

\bib{Nigel_cyl}{article}{
      author={Veitch, B.},
      author={Peake, N.},
       title={Acoustic propagation and scattering in the exhaust flow from
  coaxial cylinders},
        date={2008},
        ISSN={0022-1120},
     journal={J. Fluid Mech.},
      volume={613},
       pages={275\ndash 307},
         url={http://dx.doi.org/10.1017/S0022112008003169},
      review={\MR{2462429 (2009j:76203)}},
}

\bib{Ab_Kh_n}{article}{
      author={Veitch, Benjamin~H.},
      author={Abrahams, I.~David},
       title={On the commutative factorization of {$n\times n$} matrix
  {W}iener-{H}opf kernels with distinct eigenvalues},
        date={2007},
        ISSN={1364-5021},
     journal={Proc. R. Soc. Lond. Ser. A Math. Phys. Eng. Sci.},
      volume={463},
      number={2078},
       pages={613\ndash 639},
         url={http://dx.doi.org/10.1098/rspa.2006.1780},
      review={\MR{2288836 (2007k:47023)}},
}

\end{biblist}
\end{bibdiv}

\end{document}